\documentclass{amsart}

\usepackage{amsmath, graphicx, cite, enumerate, mathrsfs, comment, color}
\usepackage{amssymb, amsthm, amscd}

\newtheorem{theorem}{Theorem}[section]

\newtheorem{lemma}[theorem]{Lemma}
\newtheorem{proposition}[theorem]{Proposition}
\newtheorem{corollary}[theorem]{Corollary}

\theoremstyle{definition}

\theoremstyle{remark}
\newtheorem{remark}[theorem]{Remark}

\numberwithin{equation}{section}

\DeclareMathAlphabet{\mathpzc}{OT1}{pzc}{m}{it}

\newcommand{\abs}[1]{\left|#1\right|}
\newcommand{\tr}{\textup{tr}}
\newcommand{\Rm}{\textup{Rm}}
\newcommand{\Ric}{\textup{Ric}}

\newcommand{\norm}[1]{\left\| #1 \right\|}
\newcommand{\Vol}{\textup{Vol}}

\newcommand{\FS}{\textup{FS}}

\newcommand{\R}{\mathbb{R}}
\newcommand{\C}{\mathbb{C}}

\newcommand{\V}{\V}
\newcommand{\CP}{\mathbb{CP}}

\renewcommand{\H}{\mathcal{H}}
\renewcommand{\V}{\mathcal{V}}

\newcommand{\D}[2]{\frac{\partial #1}{\partial #2}}
\newcommand{\ddbar}{\partial \bar{\partial}}

\begin{document}
\title[RF on $\CP^1$-bundles over product of KE manifolds]{Ricci Flow on $\CP^1$-bundles over a Product of K\"ahler-Einstein Manifolds}
\author{Frederick Tsz-Ho Fong}
\address{Frederick T.-H. Fong, Department of Mathematics, Hong Kong University of Science and Technology, Clear Water Bay, Kowloon, Hong Kong SAR}
\email{frederick.fong@ust.hk}
\author{Hung Tran}
\address{Hung Tran, Department of Mathematics and Statistics,
	Texas Tech University, Lubbock, TX 79409}
\email{hung.tran@ttu.edu}
\maketitle

\begin{abstract}
In this paper, we study the Ricci flow on $\CP^1$-bundles over a product of K\"ahler-Einstein manifolds whose initial metric $g_0$ is constructed by the ansatz used in works by M. Wang et. al. We prove that the ansatz is preserved along the Ricci flow. Furthermore, in the K\"ahler case, we proved that Type I finite-time singularity must occur. 
\end{abstract}

\section{Introduction}

The Ricci flow is one of the central topics in geometric analysis and is one of the most extensively studied geometric flows. It was introduced by Hamilton in \cite{H3}, and has been used to resolve the long-standing Poincar\'e Conjecture by Perelman \cite{perelman1,perelman2,perelman3}. One important component in many works on the Ricci flow is to study the blow-up or convergence limit of the flow. To this end, it is then essential to understand how fast the curvature blows up as the flow approaches to the maximal existence time. Following Hamilton's classification in \cite{Hsurvey}, when the Ricci flow approaches a finite-time singularity at $T < \infty$, we say that the singularity is of Type I if there exists a constant $C > 0$ such that for any $t \in [0,T)$ we have
\[\abs{\Rm(g(t))}_{g(t)} \leq \frac{C}{T-t}.\]
Otherwise, we say that the singularity is of Type II.

Type I rescaling is a more desirable 
because the blow-up factor is essentially the inverse linear expression $\frac{1}{T-t_i}$. The rescaled limit $g_i(t) := \frac{1}{T-t_i}g\big(t_i + (T-t_i)t\big)$ would then converge in Cheeger-Gromov sense to a srhinking gradient Ricci soliton as shown in \cite{emt10}. It is therefore of strong interest in knowing under what conditions the Ricci flow would encounter Type I singularity.

In the K\"ahler case, it has recently been proved by Conlon-Hallgren-Ma in \cite{CHM25typeI} that any Ricci flow on compact K\"ahler surfaces $X$ (i.e. $\dim_\C X = 2$) with finite-time singularity $T < \infty$ and $\inf_{[0,T)}\Vol(X,g(t)) > 0$ must also encounter Type I singularity, and the blow-up limit is the gradient shrinking K\"ahler-Ricci soliton on the $\C$-bundle $\mathcal{O}(-1)$ over $\CP^1$ constructed by Feldman-Ilmanen-Knopf in \cite{fik03}. In higher complex dimensions, it was proved by Li-Tian-Zhu in \cite{LTZ24singular} that Type II singularity could occur on some Fano compactification of a semi-simple complex Lie group if it admits no K\"aher-Einstein metrics. 

In contrast, it appears that if the manifold admits some sort of good symmetry, then it is more likely that Type I singularity occurs. In \cite{Fong14}, the first-named author proved that the K\"ahler-Ricci flow on $\CP^1$-bundle $\mathbb{P}(\mathcal{O} \oplus L)$ over a K\"ahler-Einstein manifold $\Sigma$ would encounter Type I singularity if the initial metric satisfies the Calabi's $U(n)$-symmetry and the $\CP^1$-fibers of the bundle collapse as $t \to T$. The rescaled Ricci flow would converge in pointed Cheeger-Gromov's sense to the direct product $\C^{\dim_\C \Sigma} \times \CP^1$. On a $\CP^1$-bundle $\mathbb{P}(\mathcal{O} \oplus\mathcal{O}(-k))$ over $\CP^{n-1}$ where the initial K\"ahler metric has Calabi's $U(n)$-symmetry, Song proved in \cite{Song15} that Type I singularity must occur when the exceptional divisor contracts to a point as $t \to T$. Subsequently, Guo and Song also proved in \cite{GS16} that the Type I parabolic blow-up limit along the exceptional divisor is the unique $U(n)$-complete shrinking K\"ahler-Ricci soliton on $\C^n$ blown-up at a point constructed in \cite{fik03}. 

Recently, Jian-Song-Tian \cite{jst23} generalized the above-mentioned works by studying the K\"ahler-Ricci flow on higher rank projective bundle $\mathbb{P}(\mathcal{O}_Z \oplus L^{\oplus m})$, $m \geq 1$, over a K\"ahler-Einstein manifolds $Z^n$. Assuming Calabi's $U(m+n)$-symmetry, they proved that Type I singularity must occur, and the blow-up limit near the singularity is the shrinking K\"ahler-Ricci soliton with Calabi symmetry on the total space of the vector bundle $L^{\oplus m}$ over $Z^n$. In addtion, there have been several recent developments in precise understanding of Ricci flows on manifolds with different types of symmetry: \cite{Appleton23}($U(2)$-metric on $T\mathbb{S}^2$), \cite{Francesco20}($SU(2)$-metric on $\mathbb{R}^4$),\cite{wu25asymptotic} ($SO(n)$-metric on $\mathbb{R}^{n+1}$), \cite{IKS25} (multiple-warped product), and the aforementioned \cite{CHM25typeI} (compact K\"{a}hler surfaces with non-collapsed finite time singularities).

It is not a coincidence that many of the above-mentioned works study the K\"ahler-Ricci flow on underlying manifolds which admit some bundles or fibration structures, as these underlying manifolds appear frequently in the context of minimal model program proposed by Song-Tian in \cite{ST17KRf}. In fact, in complex dimension 3, it was shown by Tosatti-Zhang in \cite{TZ15} that if the K\"ahler-Ricci flow encounters finite-time singularity, the underlying manifold must admit a Fano fibration. Apart from the above-mentioned references, there are other works investigating the finite-time singularity of a K\"ahler-Ricci flow, sharing similar contexts as in this paper, such as \cite{SW11,SW13i,SW13ii,SY12} and the references therein.

To enrich the perspective of the singularity analysis of Ricci flow with symmetry, this article studies the Ricci flow on a particular family of metrics, with the ansatz used by Wang-Wang \cite{WW98} and Dancer-Wang \cite{DW11coho} among others. Their works were about the construction of Einstein metrics and Ricci solitons on manifolds foliated by equidistant hypersurfaces that 
are circle bundles over products of K\"ahler-Einstein manifolds. 
We will recall fundamental components of the ansatz in detail in Section \ref{sect:DW}. Roughly speaking, this ansatz is built upon a submersion of a Riemannian manifold $\pi : P \to N$ onto a \emph{product} of K\"ahler-Einstein manifolds $N = N_1 \times \cdots \times N_r$ with circle fibers, such that there exists a contact $1$-form $\eta$ on $P$. One then defines a metric $g$ on the product $[0,s_0] \times P$ with metric of the form given by \eqref{eq:DW} and some closing conditions described in Subsection \ref{subsect:closing}. This ansatz generalizes the Calabi's $U(n)$-symmetry whose base manifold is one single homogeneous K\"ahler-Einstein manifold.
\newpage
Working on the above-mentioned ansatz, we study the Ricci flow and its singularity analysis, and obtain the following:

\vskip 0.3cm

\noindent\textbf{Main Result.} \emph{Let $M \to N$ be a $\CP^1$-bundle over $N = N_1 \times \cdots \times N_r$ which is a product of compact K\"ahler-Einstein manifolds $(N_i, J_i, g_i)$ such that the initial metric $g_0$ is constructed by the ansatz described in Section \ref{sect:DW}. Then, the ansatz is preserved under the Ricci flow (Proposition \ref{preservingansatz}). Furthermore, in the K\"ahler case, the Ricci flow solution must encounter Type I finite-time singularity (Theorem \ref{thm:TypeI}).}

\vskip 0.3cm

Our results generalize some former works such as \cite{Fong14,Song15,GS16,Maximo14} by keeping the fibers to be $\CP^1$ but allowing the base manifold to a \emph{product} of K\"ahler-Einstein manifolds, whereas these former work required the base manifold to be \emph{one} K\"ahler-Einstein manifold. Such a generalization is also different from the recent work \cite{jst23} which studies the \emph{higher-rank} bundles over \emph{one} K\"ahler-Einstein manifold.

It would also be interesting, as in the above-mentioned former works, to identify the blow-up singularity model by studying its pointed Cheeger-Gromov limit. If the singularity is caused by the collapsing of $\CP^1$-fibers, then by the same argument as in \cite{Fong14} one can prove that the blow-up model is the product $\CP^1 \times \C^n$ where $n = \dim_\C N$. However, if the finite-time singularity is caused by the contraction of some (but not all) of the factor $N_i$'s, it is not straight-forward to find out its blow-up limit. We will outline the possibilities and the detail will be discussed in forthcoming work.

\vskip 0.2cm
\noindent \textbf{Acknowledgement.} The first-named author is partially supported by the General Research Fund \#16305625 of the Hong Kong Research Grants Council. Part of this material is based upon work supported by the National Science Foundation under Grant No. DMS-1928930, while the second-named author was in residence at the MSRI in Berkeley, California, during the Fall semester of 2024. We would also like to thank the Vietnam Institute for Advanced Study in Mathematics (VIASM) for hosting us in the summer of 2025.

\section{An Ansatz with Equidistant Hypersurfaces} 

In this section, we give a definition and describe properties of a general ansatz with Calabi's symmetry as its special case. This particular setting was considered in \cite{WW98,DW11coho}, among others, to construct Einstein metrics and Ricci solitons. 

\subsection{The ansatz}
\label{sect:DW}

The ansatz (c.f. \cite[Section 4]{DW11coho}) used in \cite{WW98} and \cite{DW11coho} is described as follows. Let
\[\big(N_i^{2n_i}, J_{N_i}, g_{N_i}, \omega_{N_i}\big), \quad 1 \leq i \leq r\]
be compact K\"ahler-Einstein manifolds of real dimension $2n_i$ with $\Ric(g_{N_i}) = k_i g_{N_i}$ for constant $k_i \in \R$. Here $J_{N_i}$ is the complex structure of $N_i^{2n_i}$, and $\omega_{N_i}$ is the K\"ahler form defined by
\[\omega_{N_i}(X,Y) = g_{N_i}(J_{N_i}X, Y)\]
for any smooth vector fields $X, Y$. We denote the product manifold of these $N_i$'s by
\[N := N_1^{2n_1} \times \cdots \times N_r^{2n_r}.\]

Next we consider another Riemannian manifold $P$ of real dimension 
\[\dim P = 1+\sum_{i=1}^r \dim N_i\] such that there exists a submersion $\pi : P \to N$ with circle fibers, and there exists an $1$-form $\eta$ on $P$ satisfying
\[d\eta = \sum_{i=1}^r q_i \pi_i^*\omega_{N_i}\]
for integers $q_i \in \mathbb{Z} \backslash \{0\}$. We let $\zeta$ in the kernel of $\pi_*$ 
be the dual vector field of $\eta$ such that $\eta(\zeta) = 1$. A particular case is given by the Hopf fibration of an odd-dimensional sphere over the complex projective space.  

Now consider a smooth manifold $M$ with a dense subset $M_0$ which is diffeomorphic to $I \times P$ where $I$ is an open interval. For each $s \in I$, we consider a metric $g_s$ on $\{s\} \times P$ of the form
\[g_s = H^2(s) \eta \otimes \eta + \sum_{i=1}^r F_i^2(s) \pi_i^*g_{N_i}\]
where $\pi_i : P \to N_i$ is the composition of $\pi : P \to N$ composed with the $i$-th canonical projection map $N \to N_i$. The metric on $M_0$ is  in the form of
\begin{equation}
\label{eq:DW}
g = ds^2 + g_s = ds^2 + H^2(s) \eta \otimes \eta + \sum_{i=1}^r F_i^2(s) \pi_i^*g_{N_i}
\end{equation}
The complete metric on $M$ is obtained by smooth compactification at end points of $I$ for which the closing condition will be discussed later. \\

Let's examine this construction in more details. Denote $E := \ker(\eta : TP \to \R)$, then $TM_0$ can then be decomposed into
\[TM = \textup{span}\left\{\D{}{s}\right\} \oplus \textup{span}\left\{\zeta\right\} \oplus E.\]
Note that $\zeta$ may not be an integrable vector field. With some abuse of notations, we still denote the following submersion map from $M_0 = I \times P \to N$ by $\pi$, and similarly for each $\pi_i$, then the vertical and horizontal distributions with respect to $\pi : M_0 \to N$ is given by
\begin{align*}
	\mathcal{V} & := \ker\pi_* = \textup{span}\left\{\D{}{s}, \zeta\right\}\\
	\mathcal{H} & := \mathcal{V}^{\perp_g} = E.	
\end{align*}
Moreover, each $F_i$ is the conformal factor of the submersion $\pi_i$ to the fixed metric $(N_i, g_i)$ while $H$ is associated with the length of the circle fiber. By continuous extension to $M$, they can be considered as smooth non-negative functions from $M$ to $\mathbb{R}$ independent of coordinates. \\   

We can define an almost complex structure $J$ on $TM_0$ by, for each $X \in \mathcal{H}$,
\begin{align*}
	J\D{}{s} & = -\frac{1}{H(s)}\zeta, \text{ and }	J\zeta = H(s)\D{}{s}\\
	JX & = \sum_{i=1}^r J_{N_i}\big((\pi_i)_*(X)\big).
\end{align*}
It is immediate that $g$ is Hermitian with respect to $J$ and the K\"{a}hler form $\omega$ of $(M_0, J, g)$ is defined by
\[\omega(\cdot, \cdot) = g(J\cdot, \cdot).\]
By observing that $\eta(X) = 0$ for any $X \in \H$ and that $\pi_*(\zeta) = 0$, we obtain
\[\omega = H\eta \wedge ds + \sum_{i=1}^r F_i^2(s) \pi_i^*\omega_{N_i}.\]
Henceforth, by the fact that $H$ depends only on $s$ so that $dH \wedge ds = 0$, we have
\begin{align*}
	d\omega & = dH \wedge \eta \wedge ds + Hd\eta \wedge ds + \sum_{i=1}^r dF_i^2 \wedge \pi_i^* \omega_{N_i}\\
	& = \sum_{i=1}^r H(s) q_i\pi_i^*\omega_{N_i} \wedge ds + \frac{d}{ds}F_i^2(s) ds \wedge \pi_i^*\omega_{N_i}.
\end{align*}
Therefore, $M_0$ is K\"ahler with respect to $J$ iff
\begin{equation}
	\label{eq:Kahler}
	q_i H(s) = \frac{dF_i^2(s)}{ds} \;\; \forall i = 1, \cdots, r.	
\end{equation}

\subsection{Closing conditions}
\label{subsect:closing}
Suppose that $I$ has the left end-point $s_0$ where \[\lim_{s\rightarrow s_0^+} H(s)=0< \lim_{s\rightarrow s_0^+}F_i(s).\] The completeness of the metric imposes certain conditions on $H$ and each $F_i$ approaching this point. Let $\theta$ be a (local) $\mathbb{S}^1$-fiber coordinate of submersion $\pi : P \to N$, and write $\eta$ locally as $\eta = d\theta + \alpha$ for some $1$-form $\alpha$ independent of $\theta$. By expressing the metric using polar coordinates $x = r\cos\theta$, $y = r\sin\theta$, where $r(s)$ is a function of $s$ obtained by solving the ODE:
\[\frac{1}{r}dr = \frac{1}{H(s)}ds, \quad \text{and} \quad r(s_0) = 0.\]
Then, the metric can be written as
\begin{align*}
	g & = \frac{H(s)^2}{r^2}\left(dr^2 + r^2(d\theta + \alpha) \otimes (d\theta + \alpha)\right) + \sum_j F_j(s)^2 \pi_j^*g_{N_j}\\
	& = \frac{H(s)^2}{x^2+y^2}\left(dx^2 + dy^2 + (x\,dy - y\,dx) \cdot \alpha + (x^2+y^2) \alpha \otimes \alpha\right) + \sum_j F_j(s)^2 \pi_j^*g_{N_j}.
\end{align*}
Let $\bar{H}(r)$ and $\bar{F}_j(r)$ be functions such that
\[\bar{H}(r(s)) = H(s) \text{ and } \bar{F}_j(r(s)) = F_j(s),\]
then in order to extend the metric of $I \times P$ to the left end-point $s_0 \in I$, we require that there exist smooth functions $\phi$ and $\psi_j$ of $r$ defined near $r = 0$ and with $\phi(0) > 0$ and $\psi_j(0) > 0$ such that
\[\frac{\bar{H}(r)}{\sqrt{x^2+y^2}} = \phi(x^2+y^2) = \phi(r^2), \text{ and } \bar{F}_j(r) = \psi_j(x^2+y^2)\]
such that $\phi(0), \psi_j(0) > 0$. By considering Taylor expansion, the above is equivalent to the conditions that
\[H'(s_0)=1, \text{ and } H^{(2k)}(s_0) = F_i^{(2k+1)}(s_0) = 0 \text{ for any $k \geq 0$.}\]
The condition is similar if $s_1$ is the right end-point of $I$, except that the first condition is replaced by $H'(s_1) = -1$.

\subsection{Special case: Calabi symmetry}
When the base manifold $N$ consists of only one factor, i.e. $r = 1$, then for each $s \in I$, the submersion $\pi : (P, g_s) \to (N, F(s)^2g_N)$ is a Riemannian submersion, and the submersion $\pi : (I \times P, g) \to (N, g_N)$ is a conformal submersion.

One important special case of the ansatz \eqref{eq:DW} is given by metrics under Calabi's ansatz. It can be obtained by taking $P = \mathbb{S}^{2n-1}/\mathbb{Z}_k := L^{2n-1}(k)$, commonly called a lens space, with the round metric, and $N = \CP^{n-1}$ with the Fubini-study metric $g_{\FS}$. The submersion $\pi : \mathbb{S}^{2n-1} / \mathbb{Z}_k \to \CP^{n-1}$ is then the canonical projection. Take $I = (0, \infty)$, then $I \times P$ is biholomorphic to $(\C^n \backslash \{0\})/\mathbb{Z}_k$. A metric $g$ defined on $(\C^n \backslash \{0\})/\mathbb{Z}_k$ is said to have \textbf{Calabi symmetry} under the global coordinates $(z_1, \cdots, z_n)$ on $\C^n\backslash\{0\}$ if it can be written in the form
\[g_{i\bar{j}} = g\left(\D{}{z_i}, \D{}{\bar{z}_j}\right) = \frac{\partial^2 u}{\partial z_i \partial \bar{z}_j}\]
where $u$ depends only $\|z\|^2 = \sum_i \abs{z_i}^2$. It is typical that we let $\rho = \log \|z\|^2$, then one can directly compute that $\sqrt{-1}\ddbar\rho = \pi^*\omega_{\FS}$. By the chain rule, one can check that
\[g_{i\bar{j}} = u_{\rho\rho} \D{\rho}{z_i}\D{\rho}{\bar{z}_j} + u_\rho \frac{\partial^2\rho}{\partial z_i \partial \bar{z}_j}.\]
By letting
\begin{align*}
	\eta & = \frac{k\sqrt{-1}}{2}\left(\bar\partial\rho - \partial\rho\right)	= \frac{k\sqrt{-1}}{2}\left(\D{\rho}{\bar{z}_i}d\bar{z}^i - \D{\rho}{z_i}dz^i\right),\\
	\zeta & = \frac{\sqrt{-1}}{k}\left(z_i\D{}{z_i} - \bar{z}_i\D{}{\bar{z}_i}\right),
\end{align*}
then we have $\eta(\zeta) = 1$ since $\D{\rho}{z_i} = \frac{\bar{z}_i}{\norm{z}^2}$, and $d\eta = k\pi^*\omega_{\FS}$ by $\sqrt{-1}\ddbar\rho = \pi^*\omega_{\FS}$. One can then check that
\[g = u_{\rho\rho}\left(d\rho \otimes d\rho + \frac{1}{k^2}\eta \otimes \eta\right) + u_\rho \pi^*g_{\FS}.\]
By the change of variables $s = \psi(\rho)$ via the ODE
\[\frac{d\psi}{d\rho} = \sqrt{u_{\rho\rho}}\]
one can rewrite $g$ in the form of
\[g = ds^2 + \frac{1}{k^2} u_{\rho\rho} \circ \psi^{-1}(s) \eta\otimes\eta + u_\rho \circ \psi^{-1}(s)\pi^*g_{\FS},\]
corresponding to ansatz \eqref{eq:DW} with
\[H(s) = \sqrt{\frac{1}{k^2}u_{\rho\rho} \circ \psi^{-1}(s)} \quad \text{ and } \quad F(s) = \sqrt{u_\rho \circ \psi^{-1}(s)}.\]
It is easy to verify that
\[\frac{d}{ds}F(s)^2 = u_{\rho\rho}\frac{d\rho}{ds} = \sqrt{u_{\rho\rho} \circ \psi^{-1}(s)}= kH(s),\]
which is exactly the K\"ahler condition.

\subsection{O'Neill Tensors}
A submersion $\pi : M \to N$ is often associated the following two important tensors introduced by O'Neill \cite{ONeill66}. For any $X, Y \in TM$, 
\begin{align*}
	A_X Y & := \V\nabla_{\H X}\H Y + \H\nabla_{\H X}\V Y\\
	T_X Y & := \H\nabla_{\V X}\V Y + \V\nabla_{\V X}\H Y
\end{align*}
where $\H$ and $\V$ denotes the horizontal and vertical projections respectively. One can check that under the ansatz \eqref{eq:DW}, we have $T \equiv 0$, i.e. the fibers of $\pi$ are totally geodesic. For the $A$ tensor, we collect the following whose proofs are given in the Appendix for completeness. Let $\sharp$ denote the vector dual to the one-form with respect to $g$. 
\begin{lemma}
	\label{lm1}
	For the ansatz \eqref{eq:DW}, the $A$-tensor is given by
	\begin{align*}
		A_X Y & = \frac{1}{2}\V[X,Y] - \sum_{i=1}^r  F_i^2\pi_i^*g_{N_i}(X,Y)\V\nabla \log F_i \\ 
		& = - \sum_{i=1}^r F_i^2\pi_i^*\omega_{N_i}(X,Y)\V(J\nabla \log F_i) - \sum_{i=1}^r  F_i^2\pi_i^*g_{N_i}(X,Y) \V\nabla \log F_i \nonumber\\
		& \hskip 0.5cm - \frac{1}{2} \V\left(d\omega(X,Y,J(\cdot))\right)^\sharp \nonumber\\
		A_X V & = -\H\left(g\big(V, A_X(\cdot)\big)^\sharp\right)\\
		A_U X & = A_U V = 0 
	\end{align*}
	for any $X, Y \in \H$ and $U, V \in \V$. 
\end{lemma}

\begin{proof}
See Appendix.	
\end{proof}

\begin{corollary}
	\label{cor:|A|^2<|grad u|^2}
	There exists a constant such that
	\begin{equation*}
		\abs{A}_g^2 \leq C\left(\sum_{i=1}^r \abs{\nabla\log F_i}_g^2 + \abs{d\omega}_g^2\right).
	\end{equation*}
	
\end{corollary}


\subsection{Computation of Curvature}
Next, we will compute the curvature tensors via Gauss, Codazzi, and Riccati equations and the second fundamental form. Regard each $P_s := \{s\} \times P$ as a hypersurface of $M$, we consider the shape operator $L_s$, as in \cite[Remark 2.18]{DW11coho}, defined by
\[L_s(X):= \nabla_X \nu \;\; \text{for any } X \in TP_s\]
where $\nu:=\frac{\partial}{\partial s}$ and $\nabla$ is the Levi-Civita connection on $M$. We will denote it simply by $L$ if it is clear from the context. For each $s$, $L_s$ can be diagonalized with eigenvalues $\frac{H'}{H}$ (multiplicity one) and $\frac{F_i'}{F_i}$ (multiplicity $2n_i$). We denote
\[L'_s = \nabla_{\nu}L_s.\]
Consequently, the eigenvalues of $L_s$ and $L'_s$ are given by 
\begin{align*}
	L_s & : \frac{H'}{H}, \underbrace{\frac{F_i'}{F_i}}_{\text{multiplicity $2n_i$}}\\
	L'_s & : \frac{H''}{H}-\left(\frac{H'}{H}\right)^2, \underbrace{\frac{F_i''}{F_i}-\left(\frac{F_i'}{F_i}\right)^2}_{\text{multiplicity $2n_i$}}
\end{align*}
Then, for the trace $\tr$ is taken with respect to $g_s$, we have
\begin{align*}
	\tr{L_s} &=\frac{H'}{H}+\sum_i 2n_i\frac{F'_i}{F_i},\\
	\tr{L^2_s} &=\left(\frac{H'}{H}\right)^2+ \sum_i 2n_i\left(\frac{F'_i}{F_i}\right)^2,\\
	\tr{L'_s} &=\frac{H''}{H}-\left(\frac{H'}{H}\right)^2+\sum_i 2n_i\left(\frac{F''_i}{F_i}-\left(\frac{F'_i}{F_i}\right)^2\right).
\end{align*}

By the Gauss, Codazzi, and Riccati equations, the Ricci curvature $\Ric$ of $(M, g)$ is totally determined by that of $(P, g_s)$ and the shape operator $L_s$. Precisely, for tangential vectors $X$ and $Y$ in $TP_s$, we have
\begin{align}
	\Ric(X,Y) &=\widetilde{\Ric}_s(X,Y)-\tr(L_s)g_s(L X, Y)-g_s({L'}(X),Y), \nonumber\\
	\label{Rccomp}
	\Ric(X, \nu) &=-\nabla_X \tr(L_s)-g_s(\delta L, X),\\ 
	\Ric(\nu, \nu) &= -\text{tr}({L'})-\text{tr}(L^2).\nonumber
\end{align}
Here $\widetilde{\Ric}_s$ denotes the Ricci curvature of $(P, g_s)$, which we may abbreviate by $\widetilde{\Ric}$. By observing that the eigenvalues of $L_s$ depends only on $s$, we have $\Ric(u,\nu) = 0$ for any $u \in TP_s$. From the expressions of $\tr L^2$ and $\tr L'$, we can easily see that
\[\Ric(\nu,\nu)  = -\frac{H''}{H} - \sum_i 2n_i\frac{F''_i}{F_i}.\]
To compute the $\Ric(u,w)$-components, we need to first find out the Ricci curvature $\widetilde{\Ric}$ for $(P, g_s)$. Recall that $(P, g_s)$ is a Riemannian submersion over the product manifold $(N=N_1 \times \cdots \times N_r, g_N=\sum_i F_i^2 g_{N_i})$, and each fiber is the trajectory of a Killing vector field $\zeta$ which is the dual of $\eta$.  

The following calculation is an application of the submersion toolkit restricted to  $\pi : P_s \to N$. Denote $\tilde{A}$ be the corresponding O'Neill's $A$-tensor with respect to this setup. Denote $\eta^\ast:=H\eta$ and $\zeta^\ast$ be the dual vector field of $\eta^\ast$. Then the vertical and horizontal distributions in $TP_s$ with respect to this submersion are given by $\tilde{\V} = \text{span}\{\zeta^*\}$ and $\tilde{\H} := \tilde{\V}^{\perp} = TN_1 \oplus \cdots \oplus TN_r$.
\begin{lemma}
	\label{computeAmul}
	For any $Y \in \Gamma^\infty(\tilde{\H})$ and $X_i \in \pi^\ast(TN_i)$ 
	we have
	\begin{align*}
		\tilde{A}_{X_i} Y & = -\frac{q_iH}{2F_i^2}g_s(JX_i, Y)\zeta^*\\
		\tilde{A}_{X_i}\zeta^* & = \frac{q_iH}{2F_i^2}JX_i\\
	\end{align*}
\end{lemma}

\begin{proof}
See Appendix.	
\end{proof}

\begin{proposition}
	\label{Rcdeformed1} The sectional curvature $\tilde{K}$ and the Ricci curvature $\widetilde{\Ric}$ of $(P, g_s)$ are given as follows. For any horizontal vectors $X_i, Y_i, Z_i\in TN_i$ such that $\{X_i, Z_i\}$ are linearly independent:
	\begin{align*}
		\tilde{K}(X_i, Z_i) &= K_{N_i}(X_i,Z_i)-\frac{3q_i^2 H^2}{4F_i^4} g_s (JX_i,  Z_i)^2,\\
		\tilde{K}(X_i, \zeta^\ast) &= \frac{q_i^2 H^2}{4F_i^4}\abs{X_i}_{g_s}^2,\\
		\widetilde{\Ric}(X_i, Y_i) &= \left(\frac{k_i}{F_i^2}-\frac{q_i^2H^2}{2F_i^4}\right)g_s(X_i, Y_i)\\\
		\widetilde{\Ric}(\zeta^\ast, \zeta^\ast) &=\sum_i n_i \frac{q_i^2 H^2}{2F_i^4}.
	\end{align*}
	Here $K_{N_i}$ and $\Ric_{N_i}$ are the sectional and Ricci curvature of $(N_i, g_{N_i})$.  The other cross terms of the Ricci tensor are all vanishing. 
\end{proposition}

\begin{proof}
See Appendix.	
\end{proof}

Combining \eqref{Rccomp}, Proposition \ref{Rcdeformed1} and the eigenvalues of $L$ and $L'$, we can determine the Ricci curvature $\Ric$ of $M$ when restricted to $X_i, Y_i \in TN_i$. 

To summarize, we have proved that:

\begin{proposition}
	\label{prop:Ricci_Riemannian}
	Under the ansatz \eqref{eq:DW} defined in Section \ref{sect:DW}, the Ricci curvature $(M,g)$ is given as follows: for any $X_i, Y_i \in TN_i$, we have
	\begin{align*}
		& \Ric(X_i,Y_i)\\
		& = \widetilde{\Ric}_s(X_i,Y_i) - \left(\frac{F_i'}{F_i}\left(\frac{H'}{H}+\sum_j 2n_j\frac{F'_j}{F_j}\right) + \frac{F_i''}{F_i}-\left(\frac{F_i'}{F_i}\right)^2\right)g_s(X_i,Y_i) \nonumber\\
		& = \left(\frac{k_i}{F_i^2} - \frac{F_i'}{F_i}\left(\frac{H'}{H}+\sum_j 2n_j\frac{F'_j}{F_j}\right) - \frac{F_i''}{F_i}+\left(\frac{F_i'}{F_i}\right)^2 - \frac{q_i^2H^2}{2F_i^4}\right)F_i^2 g_{N_i}(X_i,Y_i) \nonumber
	\end{align*}
	Other components of the Ricci curvature are given by
	\begin{align*}
		\Ric(\zeta^*,\zeta^*) & = \sum_i n_i\frac{q_i^2 H^2}{2F_i^4} - \frac{H'}{H}\left(\frac{H'}{H}+\sum_i 2n_i\frac{F'_i}{F_i}\right) - \frac{H''}{H} + \left(\frac{H'}{H}\right)^2\\
		& = \sum_i n_i \frac{q_i^2 H^2}{2F_i^4} - \frac{H'}{H}\sum_i 2n_i\frac{F'_i}{F_i} - \frac{H''}{H}  \nonumber \\
		\Ric(\nu, \nu) &= -\frac{H''}{H} - \sum_i 2n_i\frac{F''_i}{F_i}\\
		\Ric(X_i,\zeta^*) & = \Ric(X_i, \nu) = \Ric(\nu,\zeta^*) = 0.
	\end{align*}
\end{proposition}

\begin{proof}
See Appendix.	
\end{proof}

Next, we'll observe that the K\"{a}hler condition will simplify these expression. For any scalar function $u(s)$ depending only on $s$, its norm of gradient and the Laplacian are given as follows:
\begin{align*}
	\abs{\nabla u}^2 & = u'(s)^2\\
	\Delta u 
	& = u'' + \left(\frac{H'}{H} + \sum_{i=1}^r 2n_i \frac{F_i'}{F_i}\right) u'.
\end{align*}
Recall from \eqref{eq:Kahler} that $(M,g)$ is K\"ahler if and only if for each $i$
\[q_i H(s) = 2F_iF_i'(s).\]
Then one observes
\begin{align*}
	\abs{\nabla \log F_i}^2 & = \left(\frac{F_i'}{F_i}\right)^2 = \frac{q_i^2 H^2}{4F_i^4},\\
	\frac{H'}{H} & = \frac{F_i'}{F_i} + \frac{F_i''}{F_i'},\\
	\Delta \log H & = \frac{\Delta H}{H} - \frac{\abs{\nabla H}^2}{H^2} = \frac{H''}{H} + \frac{H'}{H}\left(\tr L - \frac{H'}{H}\right)\\
	& = \frac{H''}{H} + \frac{H'}{H}\sum_i 2n_i \frac{F_i'}{F_i},\\
	\Delta \log F_i & = \frac{F_i''}{F_i} + \frac{F_i'}{F_i}\left(\tr L - \frac{F_i'}{F_i}\right).
\end{align*}
Consequently, we have
\begin{equation*}
	-\Delta\log H + \sum_i 2n_i \abs{\nabla\log F_i}^2 = -\frac{H''}{H} - \sum_i 2n_i \frac{F_i''}{F_i}
\end{equation*}
which is exactly the $\Ric(\nu,\nu)$ component in Proposition \ref{prop:Ricci_Riemannian}. Thus, the Ricci tensor in Proposition \ref{prop:Ricci_Riemannian} can be simplified to, for where $X_i, Y_i \in \pi^\ast(TN_i)$.
\begin{align}
	\label{RcKah1}
	\Ric(\nu,\nu) & =\Ric(\zeta^*,\zeta^*) = -\Delta \log H + \sum_i 2n_i \abs{\nabla\log F_i}^2\\
	\Ric(X_i, Y_i) & = \left(\frac{k_i}{F_i^2} - \frac{F_i'}{F_i}\cdot \tr L - \frac{F_i''}{F_i}+\left(\frac{F_i'}{F_i}\right)^2 - 2\abs{\nabla\log F_i}^2\right)g(X_i,Y_i)\nonumber\\
	\label{RcKah2}
	& = \left(\frac{k_i}{F_i^2} - \Delta \log F_i - 2\abs{\nabla\log F_i}^2\right)F_i^2 g_{N_i}(X_i,Y_i)\\
	& = \left(k_i - \frac{1}{2}\Delta F_i^2\right)g_{N_i}(X_i,Y_i). \nonumber
\end{align}

\section{Reduction of Ricci flow to ODEs}

Proposition \ref{prop:Ricci_Riemannian} shows that, for the ansatz described in Section \ref{sect:DW}, the Ricci curvature tensor is of the same form as the metric:
	\[\Ric(g) = \alpha(s)\,ds^2 + \beta(s)\,\eta \otimes \eta + \sum_{i=1}^r \gamma_i(s) \pi_i^* g_{N_i}\]
	for scalar functions $\alpha$, $\beta$ and $\gamma_i$ of $s$. This suggests that the ansatz is preserved along the Ricci flow and this section is devoted to prove that statement. 

To this end, we consider an evolving family of metrics of the form:
\[g(\sigma, t) = a(\sigma, t)^2\,d\sigma^2 + h(\sigma,t)^2\eta \otimes \eta + \sum_{i=1}^r f_i(\sigma,t)^2\pi_i^*g_{N_i},\]
where $\sigma$ is the spatial parameter independent of $t$. This form of metric is related to the ansatz \eqref{eq:DW} under the change of variables:
\begin{align*}
	ds & =a(\sigma,t)\,d\sigma, & \sigma(s_0) & = 0\\
	H(s,t) & = h(\sigma(s),t), & F_i(s,t) & = f_i(\sigma(s),t)	
\end{align*}
Indeed, $F_i$ and $f_i$ and, similarly, $H$ and $h$ are just representation of the same functions with respect to different coordinates. By the chain rule, the derivatives of $H$ and $h$ are related by
\begin{align*}
	H_s & = \frac{h_\sigma}{a}\\
	H_{ss} & = \frac{h_{\sigma\sigma}}{a^2} - \frac{h_\sigma a_\sigma}{a^3}
\end{align*}
and there are similar identities for $F_i$ and $f_i$. With this change of variables, the Ricci curvature is given by, while all other cross-terms vanish,
\begin{align*}
	\Ric(X_i,Y_i) & = \left\{\frac{k_i}{f_i^2} - \frac{(f_i)_\sigma}{a f_i}\left(\frac{h_\sigma}{a h}+\sum_i 2n_j\frac{(f_j)_\sigma}{a f_j}\right)\right.\\
	& \hskip 0.5cm \left. - \frac{(f_i)_{\sigma\sigma}}{a^2 f_i} + \frac{(f_i)_\sigma a_\sigma}{a^3 f_i} +\left(\frac{(f_i)_\sigma}{a f_i}\right)^2 - \frac{q_i^2h^2}{2f_i^4}\right\}g_s(X_i,Y_i),\\
	\Ric(\zeta^*,\zeta^*) & = \sum_i n_i\frac{q_i^2 h^2}{2f_i^4} - \frac{h_\sigma}{a h}\sum_i 2n_i \frac{f_\sigma}{a f_i} - \frac{h_{\sigma\sigma}}{a^2 h} + \frac{h_\sigma a_\sigma}{a^3 h},\\
	\Ric(\nu,\nu) & = - \frac{h_{\sigma\sigma}}{a^2 h} + \frac{h_\sigma a_\sigma}{a^3 h} - \sum_i 2n_i \left(\frac{(f_i)_{\sigma\sigma}}{a^2 f_i} - \frac{(f_i)_\sigma a_\sigma}{a^3 f_i}\right).
\end{align*}

Taking the time derivative of the evolving metric, one obtains
\[\D{g}{t} = 2a\dot{a} \,d\sigma^2 + 2h\dot{h} \eta \otimes \eta + \sum_i 2f_i\dot{f}_i\pi_i^*g_{N_i}\]
where $\dot{a}$, $\dot{h}$ and $\dot{f_i}$ denotes the partial derivatives by $t$ fixing $\sigma$.

\begin{proposition}
	\label{preservingansatz}
	Let $(M, g(0))$ be a closed Riemannian manifold constructed in Section \ref{sect:DW}. Along the K\"ahler-Ricci flow, the structure of the ansatz is preserved so that on $M_0 = I \times P$ the evolving metric is of the form
	\begin{align*}
		g(t) &= a^2(\sigma,t)\,d\sigma^2 + h(\sigma,t)^2\,\eta\otimes\eta + \sum_{i=1}^r f_i(\sigma,t)^2\pi_i^* g_{N_i}\\
		&= ds^2 + H(s, t)\eta\otimes\eta+  \sum_{i=1}^r F_i(s,t)^2\pi_i^* g_{N_i}.
	\end{align*}
	Furthermore, each conformal factor satisfies the heat-type equation
	\begin{equation}
		\label{eq:heat}
		\D{f_i^2}{t}  = \Delta_{g(t)}f_i^2 - 2k_i=\D{F_i^2}{t}  = \Delta_{g(t)}F_i^2 - 2k_i
	\end{equation}
\end{proposition}
\begin{proof}
	The preservation of the ansatz is a direct consequence of the existence theorem of ODEs.
	Let $a(\sigma, t)$, $h(\sigma,t)$ and $f_i(\sigma,t)$ be the local solutions to the ODE system
	\begin{align*}
		\dot{a} & = \frac{h_{\sigma\sigma}}{a h} -  \frac{h_\sigma a_\sigma}{a^2 h} + \sum_i 2n_i \left(\frac{(f_i)_{\sigma\sigma}}{a f_i} - \frac{(f_i)_\sigma a_\sigma}{a^2 f_i}\right)\\ 
		\dot{h} & = -\sum_i n_i\frac{q_i^2 h^3}{2f_i^4} +\sum_i 2n_i \frac{h_\sigma f_\sigma}{a^2 f_i} + \frac{h_{\sigma\sigma}}{a^2} - \frac{h_\sigma a_\sigma}{a^3} \\
		\dot{f_i} & = -\frac{k_i}{f_i} + \frac{(f_i)_\sigma}{a}\left(\frac{h_\sigma}{a h}+\sum_j 2n_j\frac{(f_j)_\sigma}{a f_j}\right)\\
		& \hskip 0.5cm + \frac{(f_i)_{\sigma\sigma}}{a^2} - \frac{(f_i)_\sigma a_\sigma}{a^3} - \frac{{(f_i)_\sigma}^2}{a^2 f_i} + \frac{q_i^2h^2}{2f_i^3}. \nonumber
	\end{align*}
	 By the calculation above, the metric 
	\begin{align*}
		g(t) & = a(\sigma, t)^2 d\sigma^2 + h(\sigma,t)^2 \eta \otimes \eta + \sum_{i=1}^r f_i(\sigma,t)^2\pi_i^*g_{N_i}\\
		& = ds^2 + H(s,t)^2 \eta \otimes \eta + \sum_{i=1}^r F_i(s,t)^2 \pi_i^*g_{N_i}
	\end{align*}
	solves the Ricci flow equation on $M_0 = I \times P$. Since the solution to the Ricci flow on the whole compact manifold $M$ is unique, $g(t)$ must coincide with the Ricci flow solution when restricted on $M_0$. Hence, the ansatz described in Section \ref{sect:DW} is preserved.\\
	
	As the ansatz is preserved, we then have for any $X_i,Y_i \in \pi^\ast(TN_i)$,
	\[\D{g}{t}(X_i,Y_i)  = -2\Ric(X_i,Y_i).\]
	Applying equation (\ref{RcKah2}) yields
	\begin{align*}
		\D{}{t}f_i(\sigma, t)^2 g_{N_i}(X_i,Y_i) & = \left(\Delta f_i^2 - 2k_i\right)g_{N_i}(X_i,Y_i)\\
		\implies \D{}{t}f_i^2 & = \Delta f_i^2 - 2k_i
	\end{align*}
	Here we made use of the fact that $\Delta_{g(t)}$ is a spatial operator, so $\Delta f_i^2$ and $\Delta F_i^2$ are equal at the corresponding values of $\sigma$ and $s$. 
\end{proof}

Denote $u := f_j^2$, then by the well-known Bochner formula, the gradient term $\abs{\nabla u}$ would satisfy
\begin{equation}
	\label{eq:gradu}
	\D{}{t}\abs{\nabla u}^2 = \Delta\abs{\nabla u}^2 - 2\abs{\nabla\nabla u}^2
\end{equation}
In particular, $\abs{\nabla f_j^2}$ is uniformly bounded by the parabolic maximum principle.

\section{Type I Singularity}
According to Hamilton's classification \cite{Hsurvey}, a Ricci flow solution $g(t)$ which encounters finite-time singularity at $T < \infty$ can be classified into the two types:
\begin{itemize}
\item Type I singularity: if $\displaystyle{\lim_{t \to T^-}\sup_M (T-t)\abs{\Rm} < \infty}$
\item Type II singularity: if $\displaystyle{\lim_{t \to T^-}\sup_M (T-t)\abs{\Rm} = \infty}$	
\end{itemize}
One goal of this section is to establish that the singularity of the K\"ahler-Ricci flow under ansatz \eqref{eq:DW} must encounter Type I singularity.

\subsection{Riemann curvature in the K\"ahler case} To achieve the goal of this section, we first need to determine the Riemann curvature of $M$ along the horizontal components in the K\"ahler case, which will be needed in the proof.

The second fundamental of $P_s := \{s\} \times P$ in $M = I \times P$ is defined by:
\[h(u,w) := g(\nabla_u w, \nu), \forall u,v \in TP_s\]
where $\nu := \D{}{s}$ is the unit normal vector of $P_s$ in $M$. When restricted to horizontal vectors, the second fundamental is given by:

\begin{lemma}
\label{lma:h}
For any $X, Y \in \H$, we have
\[h(X,Y) = -\sum_i \left(\D{}{s}\log F_i\right) F_i^2\pi_i^* g_{N_i}(X,Y)\]
\end{lemma}

\begin{proof}
For any $X,Y \in \H$, we have from Lemma \ref{lm1} 
\begin{align*}
h(X,Y) & = g(\nabla_X Y, \nu) = g(A_XY, \nu)\\
& = -\sum_i F_i^2 \left\{\pi_i^*\omega_{N_i}(X,Y) g\left(J\nabla\log F_i, \nu\right) + \pi^*_i g_{N_i}(X,Y) g\left(\nabla\log F_i, \nu\right)\right\}\\
& \hskip 1cm -\frac{1}{2}d\omega(X,Y,J\nu)\\
& = -\sum_i F_i^2 \pi^*_i g_{N_i}(X,Y) g\left(\nabla\log F_i, \nu\right) -\frac{1}{2}d\omega(X,Y,J\nu)\\
& = -\sum_i \left(\nu\log F_i\right) F_i^2\pi_i^* g_{N_i}(X,Y) - \frac{1}{2}d\omega(X,Y,J\nu).
\end{align*}
Since $h(\cdot,\cdot)$ and $g_{N_i}(\cdot,\cdot)$ are symmetric while $d\omega(\cdot,\cdot,J\nu)$ in anti-symmetric, we must have $d\omega(X,Y,J\nu) = 0$, completing the proof.
\end{proof}
\begin{remark}
In particular, when $r = 1$ (single base), we then have
\[h(X,Y) = -\left(\D{}{s}\log F\right)g(X,Y).\]
\end{remark}

The Riemann curvature tensor of $M_0 = I \times P$ can be expressed in terms of that of $P_s$ by the Gauss-Codazzi equations for hypersurfaces. Since $\pi : P \to N$ is a Riemannian submersion, one can then find out the Riemann curvature tensor of $P_s$ by known formulae in the literature.

The Gauss-Codazzi equations assert that for any $u,v,w,x \in TP_s = \textup{span}\{\zeta\} \oplus \H$, the Riemann curvature $\Rm$ of $(M_0,g)$ is given by
\begin{align*}
\Rm(u,v,w,x) & = \widetilde{\Rm}(u,v,w,x) + h(u,w)h(v,x) - h(v,w)h(u,x)\\
\Rm(u,v,w,\nu) & = (\tilde\nabla_u h)(v,w) - (\tilde\nabla_v h)(u,w)
\end{align*}
where $\tilde\nabla$ is the covariant derivative of $P_s$. For the purpose of this article, we only need to determine $\Rm(X,Y,Z,W)$ where $X,Y,Z,W \in \H$ when $M$ is K\"ahler.
 
Let $X,Y,Z,W \in \H$, we have
\[\Rm(X,Y,Z,W) = \widetilde{\Rm}(X,Y,Z,W) + h(X,Z)h(Y,W) - h(Y,Z)h(X,W)\]

Now we only consider the case that $(M,J,g)$ is K\"ahler, so that $q_i H = 2F_iF_i'$ for any $i$. From Lemma \ref{lma:h}, for $X, Y \in \H$, we then have
\begin{align*}
h(X,Y) & = -\sum_i F_iF_i'\pi_i^*g_{N_i}(X,Y)\\
& = -\sum_i \frac{q_i}{2}H \pi_i^*\omega_{N_i}(X,JY)\\
& = -\frac{1}{2}Hd\eta(X,JY).
\end{align*}

Similarly, from Lemma \ref{lm1} the $A$-tensor is given by:
\begin{align*}
A_XY & = \frac{1}{2}d\eta(X,Y)\zeta - \frac{1}{2}H d\eta(X,JY)\nu,& \forall X, Y \in \H.
\end{align*}

As for the $\widetilde{\Rm}$ terms, we recall that $\pi : P_s \to N$ is a Riemannian submersion, so for any vectors $\{X,Y,Z,W\}$ on $N$ which are lifted to vectors $\{\tilde{X},\tilde{Y},\tilde{Z},\tilde{W}\}$ on $P_s$, it is well-known that:
\begin{align*}
\widetilde{\Rm}(\tilde{X},\tilde{Y},\tilde{Z},\tilde{W}) & = \Rm_N(X,Y,Z,W) -\frac{1}{4}g_{P_s}\big(\V[\tilde{X},\tilde{Z}], \V[\tilde{Y},\tilde{W}]\big)\\
& \hskip 0.5cm + \frac{1}{4}g_{P_s}\big(\V[\tilde{Y},\tilde{Z}], \V[\tilde{X},\tilde{W}]\big) - \frac{1}{2}g_{P_s}\big(\V[\tilde{Z},\tilde{W}], \V[\tilde{X},\tilde{Y}]\big)
\end{align*}
By \eqref{eq:V[X,Y]}, we have
\begin{align*}
\V[\tilde{X},\tilde{Y}] & = -2\sum_{i=1}^r F_i^2 \pi_i^*\omega_{N_i}(\tilde{X}, \tilde{Y}) \V J\nabla \log F_i\\
& = -2\sum_{i=1}^r F_i^2 \pi_i^*\omega_{N_i}(\tilde{X}, \tilde{Y}) \V J\left(\D{}{s}\log F_i \cdot \D{}{s}\right)\\
& = 2\sum_{i=1}^r F_i^2 \cdot \frac{F_i'}{F_i} \cdot \frac{1}{H} \pi_i^*\omega_{N_i}(\tilde{X},\tilde{Y})\cdot \zeta\\
& = \sum_{i=1}^r q_i\pi_i^*\omega_{N_i}(\tilde{X},\tilde{Y}) \cdot \zeta & \text{(since $q_iH = \frac{\partial}{\partial s} F_i^2$)}\\
& = d\eta(\tilde{X},\tilde{Y}) \cdot \zeta
\end{align*}
Recall that $\abs{\zeta}^2 = H^2$. The above shows
\begin{align*}
& \widetilde{\Rm}(\tilde{X},\tilde{Y},\tilde{Z},\tilde{W})\\
& = \Rm_N(X,Y,Z,W) - \frac{1}{4}H^2 d\eta(\tilde{X},\tilde{Z})d\eta(\tilde{Y},\tilde{W}) + \frac{1}{4}H^2 d\eta(\tilde{Y},\tilde{Z})d\eta(\tilde{X},\tilde{W})\\
& \hskip 0.5cm - \frac{1}{2}H^2 d\eta(\tilde{Z},\tilde{W})d\eta(\tilde{X},\tilde{Y})\\
& = \sum_{i=1}^r \pi_i^* \Rm_{(N_i,F_i^2 g_{N_i})}(X,Y,Z,W) - \frac{1}{4}H^2 d\eta(\tilde{X},\tilde{Z})d\eta(\tilde{Y},\tilde{W}) +\\
& \hskip 0.5cm \frac{1}{4}H^2 d\eta(\tilde{Y},\tilde{Z})d\eta(\tilde{X},\tilde{W}) - \frac{1}{2}H^2 d\eta(\tilde{Z},\tilde{W})d\eta(\tilde{X},\tilde{Y})
\end{align*}
Since the rescaling factor $F_i(s)^2$ is independent of the coordinates of $N_i$, we have
\[\Rm_{(N_i, F_i^2,g_{N_i})} = F_i^2 \Rm_{(N_i,g_{N_i})}.\]
Hence we get
\begin{align*}
& \widetilde{\Rm}(\tilde{X},\tilde{Y},\tilde{Z},\tilde{W})\\
& = \sum_{i=1}^r F_i^2 \pi_i^* \Rm_{N_i}(X,Y,Z,W)  - \frac{1}{4}H^2 d\eta(\tilde{X},\tilde{Z})d\eta(\tilde{Y},\tilde{W}) + \frac{1}{4}H^2 d\eta(\tilde{Y},\tilde{Z})d\eta(\tilde{X},\tilde{W})\\
& \hskip 0.5cm - \frac{1}{2}H^2 d\eta(\tilde{Z},\tilde{W})d\eta(\tilde{X},\tilde{Y})
\end{align*}
where we abbreviate $\Rm_{(N_i,F_i^2 g_{N_i})}$ by $\Rm_{N_i}$. Hence, we conclude that
\begin{align*}
& \Rm(\tilde{X},\tilde{Y},\tilde{Z},\tilde{W})\\
& = \widetilde{\Rm}(\tilde{X},\tilde{Y},\tilde{Z},\tilde{W}) + h(\tilde{X},\tilde{Z})h(\tilde{Y},\tilde{W}) - h(\tilde{Y},\tilde{Z})h(\tilde{X},\tilde{W})\\
& = \sum_i F_i^2 \pi_i^* \Rm_{N_i}(X,Y,Z,W) - \frac{1}{4}H^2 d\eta(\tilde{X},\tilde{Z})d\eta(\tilde{Y},\tilde{W}) + \frac{1}{4}H^2 d\eta(\tilde{Y},\tilde{Z})d\eta(\tilde{X},\tilde{W})\\
& \hskip 0.5cm  - \frac{1}{2}H^2 d\eta(\tilde{Z},\tilde{W})d\eta(\tilde{X},\tilde{Y}) + \frac{1}{4}H^2 d\eta(\tilde{X},J\tilde{Z})d\eta(\tilde{Y},J\tilde{W})\\
& \hskip 0.5cm - \frac{1}{4}H^2 d\eta(\tilde{Y},J\tilde{Z}) d\eta(\tilde{X},J\tilde{W}).
\end{align*}
In particular,
\begin{align*}
& \Rm(\tilde{X},\tilde{Y},\tilde{Y},\tilde{X})\\
& = \sum_i F_i^2\pi_i^*\Rm_{N_i}(X,Y,Y,X) + \frac{3}{4}H^2d\eta(\tilde{X},\tilde{Y})^2 + \frac{1}{4}H^2d\eta(\tilde{X},J\tilde{Y})^2\\
& \hskip 0.5cm -\frac{1}{4}H^2d\eta(\tilde{Y},J\tilde{Y})d\eta(\tilde{X},J\tilde{X})
\end{align*}
For any $X, Y \in \H$, one can write $X = \sum_i X_i$ and $Y = \sum_i Y_i$ where $\pi_* X_i, \pi_* Y_i \in TN_i$, then
\begin{align*}
& \sum_i F_i^2\pi_i^* \Rm_{N_i}(X,Y,Y,X)\\
& = \sum_i F_i^2 \pi_i^*\Rm_{N_i}(X_i,Y_i,Y_i,X_i) \\
& = \sum_i \frac{1}{F_i^2} \pi_i^*\Rm_{N_i}(F_iX_i,F_iY_i,F_iY_i,F_iX_i).	
\end{align*}
Observing that
\[\abs{d\eta(X,Y)} \leq \sum_i \abs{q_i\pi^*_ig_{N_i}(JX,Y)} \leq \sum_i \abs{q_i}\abs{JX_i}_{g_{N_i}}\abs{Y_i}_{g_{N_i}} \leq \sum_i \frac{\abs{q_i}}{F_i^2}\abs{X}_g \abs{Y}_g,\]
we conclude that
\begin{align}
\label{eq:Rm(XYYX)}
& \Rm(\tilde{X},\tilde{Y},\tilde{Y},\tilde{X})\\
\nonumber & = \sum_i \frac{1}{F_i^2}\Rm_{N_i}(F_iX_i,F_iY_i,F_iY_i,F_iX_i) + \sum_i O\left(\frac{H^2}{F_i^4}\abs{X}_g^2\abs{Y}_g^2\right).	
\end{align}

Recall that by the K\"ahler condition $qH = \D{}{s}F_i^2$, we have
\[O\left(\frac{H^2}{F_i^4}\right) = O\left(\frac{1}{F_i^4}(F_iF_i')^2\right) = O\left(\abs{\nabla \log F_i}^2\right).\]

\subsection{Li-Yau's gradient estimates}
One key ingredient in the proof of Type I singularity is the use of Li-Yau's gradient estimates for heat equations. Given that $u$ is a positive solution satisfying the heat-type equation \eqref{eq:heat}, the Li-Yau's estimates proves that $\frac{\abs{\nabla u}^2}{u}$ is uniformly bounded. For convenience, we include the proof here:

\begin{proposition}[Li-Yau \cite{ly86}]
Suppose $u : M \times [0,T) \to (0,\infty)$ is a positive solution to the heat-type equation \eqref{eq:heat} so that \eqref{eq:gradu} holds , then there exists $C > 0$ such that
\begin{equation}
\label{eq:gradu/u}
\frac{\abs{\nabla u}^2}{u} \leq C.
\end{equation}
\end{proposition}

\begin{proof}
By straight-forward computations, we get
\begin{equation}
\label{eq:box_gradu/u}
\left(\D{}{t}-\Delta\right)\frac{\abs{\nabla u}^2}{u} = -\frac{C\abs{\nabla\nabla u}^2}{u} + \frac{2\langle \nabla\abs{\nabla u}^2, \nabla u\rangle}{u^2} + \frac{2k\abs{\nabla u}^2}{u^2} - \frac{2\abs{\nabla u}^4}{u^3}
\end{equation}
Consider the second term of the RHS:
\begin{align*}
\frac{2\langle \nabla\abs{\nabla u}^2, \nabla u\rangle}{u^2} & = \frac{2\left\langle \nabla\frac{\abs{\nabla u}^2}{u}, \nabla u\right\rangle}{u} + \frac{2\abs{\nabla u}^4}{u^3}\\
\frac{2\langle \nabla\abs{\nabla u}^2, \nabla u\rangle}{u^2} & = (2-\varepsilon)\frac{\langle \nabla\abs{\nabla u}^2, \nabla u\rangle}{u^2} + \varepsilon\frac{\langle \nabla\abs{\nabla u}^2, \nabla u\rangle}{u^2}\\
& = (2-\varepsilon)\left(\frac{\left\langle \nabla\frac{\abs{\nabla u}^2}{u}, \nabla u\right\rangle}{u} + \frac{\abs{\nabla u}^4}{u^3}\right) + \varepsilon\frac{\langle \nabla\abs{\nabla u}^2, \nabla u\rangle}{u^2}\\
& \leq (2-\varepsilon)\left(\frac{\left\langle \nabla\frac{\abs{\nabla u}^2}{u}, \nabla u\right\rangle}{u} + \frac{\abs{\nabla u}^4}{u^3}\right) + 2\varepsilon \frac{\abs{\nabla\nabla u}\abs{\nabla u}^2}{u^2}\\
& \leq (2-\varepsilon)\left(\frac{\left\langle \nabla\frac{\abs{\nabla u}^2}{u}, \nabla u\right\rangle}{u} + \frac{\abs{\nabla u}^4}{u^3}\right) + \varepsilon\left(\frac{2\abs{\nabla\nabla u}^2}{u} + \frac{\abs{\nabla u}^4}{2u^3}\right)
\end{align*}
where $\varepsilon > 0$ is to be chosen. Substituting it back to \eqref{eq:box_gradu/u}, we get
\begin{equation*}
\left(\D{}{t}-\Delta\right)\frac{\abs{\nabla u}^2}{u} \leq (-C+2\varepsilon)\frac{\abs{\nabla\nabla u}^2}{u} + (2-\varepsilon)\frac{\left\langle \nabla\frac{\abs{\nabla u}^2}{u}, \nabla u\right\rangle}{u} + \frac{2k\abs{\nabla u}^2}{u^2} -\frac{\varepsilon}{2}\frac{\abs{\nabla u}^4}{u^3}.
\end{equation*}
Choose $\varepsilon > 0$ sufficiently small so that $-C + \varepsilon < -\frac{C}{2}$ and $-2 + \varepsilon < -1$, then we have
\begin{equation}
\label{eq:box_gradu/u2}
\left(\D{}{t}-\Delta\right)\frac{\abs{\nabla u}^2}{u} \leq -\frac{C}{2}\frac{\abs{\nabla\nabla u}^2}{u} + (2-\varepsilon)\frac{\left\langle \nabla\frac{\abs{\nabla u}^2}{u}, \nabla u\right\rangle}{u} + \frac{2k\abs{\nabla u}^2}{u^2} - \frac{\varepsilon}{2}\frac{\abs{\nabla u}^4}{u^3}.
\end{equation}
If $k \leq 0$, then the maximum principle directly shows \eqref{eq:gradu/u}. Suppose $k > 0$, then by letting $Q = \frac{\abs{\nabla u}^2}{u}$, \eqref{eq:box_gradu/u2} can be written as
\[\left(\D{}{t}-\Delta\right)Q \leq (2-\varepsilon)\frac{\langle \nabla Q, \nabla u\rangle}{u} + \frac{2k}{u}Q - \frac{\varepsilon}{2u}Q^2.\]
This implies $Q_{\max}(t) := \sup_{M \times \{t\}}Q$ is decreasing if $Q_{\max}(t) > \frac{4k}{\varepsilon}$. This also proves $Q_{\max}(t)$ is uniformly bounded from above, and hence \eqref{eq:gradu/u} holds.
\end{proof}

\subsection{Schwarz's estimates}
In this subsection we are going to derive a linear lower bound for each of $F_j$'s, and prove that if $T < \infty$ is the extinction time of the K\"ahler-Ricci flow under the ansatz \eqref{eq:DW}, then there exists a constant $C > 0$ such that $g(t) \geq \frac{1}{C}(T-t)\pi^*g_{N_i}$ which is commonly called the Schwarz's type estimates. Under the ansatz \eqref{eq:DW}, it suffices to show there exists $C > 0$ such that
\[f_j^2 \geq \frac{T-t}{C}\]
for any $t \in [0,T)$ and $j = 1, \cdots, r$. Recall that $f_j^2$ satisfies the heat equation \eqref{eq:heat}

\[\D{f_j^2}{t} = \Delta_{g(t)} f_j^2 - 2k_j.\]
in the K\"ahler case. It is interesting to observe that when restricted to the boundary $\partial M_0$, the Laplacian term $\Delta_{g(t)}f_j^2$ is a topological constant independent of $t$. It can be shown as follows.

Under the change of variables by the relation:
\[\frac{1}{H(s)}ds = \frac{1}{r}dr, \quad r(s_0) = 0,\]
where $s_0$ is the left-end point of $I$. Write $\bar{H}(r(s)) = H(s)$ and $\bar{F}_j(r(s)) = F_j(s)$. Recall that the closing conditions on $\bar{H}(r)$ and $\bar{F}_j(r)$'s are that there exist smooth functions $\phi(r)$ and $\psi_j(r)$ on $(-\varepsilon,\varepsilon)$, with $\phi(0), \psi_j(0) > 0$, such that
\[\bar{H}(r) = r\phi(r^2) \quad \text{and} \quad \bar{F}_j(r) = \psi_j(r^2).\]
Using the K\"ahler condition $q_j H(s) = \D{}{s}F_j^2$, the Laplacian term $\Delta F_j^2$ is then given by:
\begin{align*}
\Delta F_j^2 & = \partial_s\partial_s F_j^2 + \left(\frac{H_s}{H} + \sum_i 2n_i \frac{(F_i)_s}{F_i}\right) \partial_s F_j^2\\
& = \partial_s(q_jH)+\left(\frac{H_s}{H} + \sum_i \frac{n_iq_iH}{F_i^2}\right)q_j H\\
& = 2q_j H_s + \sum_i \frac{n_iq_iq_jH^2}{F_i^2}\\
& = 2q_j(r\phi(r^2))_r\frac{dr}{ds} + \sum_i \frac{n_iq_iq_jr\phi^2(r^2)}{\psi_i^2(r^2)}\\
& = 2q_j\cdot\frac{\phi(r^2) + 2r^2\phi'(r^2)}{\phi(r^2)} + \sum_i \frac{n_iq_iq_jr^2\phi^2(r^2)}{\psi_i^2(r^2)}
\end{align*}
Therefore, when $r = 0$, we have $\Delta F_j^2 = 2q_j$. Similarly, we have $\Delta F_j^2 = -2q_j$ at the right end-point of $I$.

This shows when restricted on the boundary $\partial M_0$, we have
\[\D{f_j^2}{t} = \pm 2q_j - 2k_j \implies f_j(\sigma,t)|_{\partial M_0} = f_j(\sigma,0)|_{\partial M_0} - 2(\mp q_j + k_j)t.\]
Moreover, by the K\"ahler condition we have $q_j H = \D{}{s}F_j^2$ and $H > 0$ except that the end-points of $I$. As $q_j \not= 0$, $F_j^2$ is then either strictly increasing or strictly decreasing on $I$. This shows $F_j^2$ is bounded above and below by a pair of linear functions $A-Bt$ of $t$ with $A > 0$ and $B \in \R$. If any of these linear functions goes to $0$ as $t \to T$, then the flow will extinct. Therefore these linear functions are either bounded away from $0$ on $[0,T)$, or is of the form $\frac{T-t}{C}$ for some constant $C > 0$. In either case, we can conclude that

\begin{lemma}
There exists $C > 0$ such that
\begin{equation}
\label{eq:Schwarz}
f_j^2 \geq \frac{1}{C}(T-t).
\end{equation}	
\end{lemma}

\subsection{Type I singularity}

In this subsection, we prove the main result of this article that Type I singularity must occur on the K\"ahler-Ricci flow on manifolds constructed by the ansatz defined in Section \ref{sect:DW}. Our results further extend and generalize former works \cite{Fong14,Song15,GS16} in a sense that these former works require Calabi's $U(n)$-symmetry (a special case of the ansatz \eqref{eq:DW}) so that one can analyze and estimate the K\"ahler potential. These works also require the base manifold of the submersion $\pi : P \to N$ to be \emph{one} single K\"ahler-Einstein manifold. Recently, Jian-Song-Tian \cite{jst23} further generalized the aforesaid-said result to \emph{higher rank} bundles over \emph{one} K\"ahler-Einstein manifold, proving that the K\"ahler-Ricci flow with Calabi's symmetry on these bundles must also encounter Type I singularity. In the following result, we give a generalization in a different direction: that the fibers are still $\CP^1$ but we allow the base manifold to be a \emph{product} of K\"ahler-Einstein manifolds. We also adopt a different approach in analyzing the curvature, by making use of the O'Neill's tensor to reveal the structure of the curvature.

\begin{theorem}
\label{thm:TypeI}
Let $(M^{2n},J,g_0)$ be a compact K\"ahler manifold with $\dim_{\R} = 2n$ such that there exists a holomorphic submersion $M \to N$ onto the product $N$ of K\"ahler-Einstein manifolds $\{(N_i, g_{N_i}, J_{N_i})_{i=1}^r$, and there exists an open dense subset $M_0$ which is biholomorphic to $I \times P$ for some open interval $I$ and Riemannian manifold $P^{2n-1}$ such that on $I \times P$ the initial metric $g_0$ is constructed using the ansatz described in Section \ref{sect:DW}. In particular, we have
\[g_0 = ds^2 + H(s)^2 \eta \otimes \eta + \sum_{i=1}^r F_i(s)^2 \pi_i^* g_{N_i}.\]
Then, this ansatz is preserved along the K\"ahler-Ricci flow
\[\D{}{t}g(t) = -2\Ric(g(t)), \quad g(0) = g_0.\]
Suppose the solution to flow $g(t)$ encounters finite-time singularity at $T < \infty$, then the singularity is of Type I, i.e. there exists $C > 0$ such that
\[\abs{\Rm}_{g(t)} \leq \frac{C}{T-t}.\]
\end{theorem}

\begin{proof}
The first part follows from Proposition \ref{preservingansatz}. Now we prove the main result concerning the Type I singularity formation. Let $T < \infty$ be the extinction time of the flow. For any sequence $t_i \to T$, we consider the scaling factor $K_i := \sup_M \abs{\Rm}_{g(t_i)} = \abs{\Rm}(x_i,t_i)$ where $x_i \in M$, and let
\[g_i(t) := K_i g(t_i + K_i^{-1}t), \quad t \in [-K_it_i, K_i(T-t_i)).\]
From Corollary \ref{cor:|A|^2<|grad u|^2}, we have
\[\abs{A}_{g(t)}^2 \leq C\abs{\nabla \log F_j}_{g(t)}^2 = \frac{C\abs{\nabla F_j}^2_{g(t)}}{F_j^2}\]
for each $j$, so after rescaling by $K_i$, we get
\[\abs{A}_{g_i(t)}^2 = \frac{1}{K_i}\abs{A}_{g(t_i + K_i^{-1}t)}^2 \leq \frac{C\abs{\nabla F_j}^2_{g(t_i+K_i^{-1}t)}}{K_i F_j^2(t_i + K_i^{-1}t)} \leq \frac{C\abs{\nabla F_j}^2_{g(t_i+K_i^{-1}t)}}{K_i(T-t_i-K_i^{-1}t)}.\]

The last inequality follows from \eqref{eq:Schwarz}. From \eqref{eq:gradu/u}, by taking $u = F_j^2$, we get $\abs{\nabla F_j} \leq C$, and hence we conclude that
\begin{equation}
\label{eq:|A|^2<C/K(T-t)}
\abs{A}_{g_i(t)}^2 \leq \frac{C}{K_i(T-t_i - K_i^{-1}t)}.
\end{equation}

Therefore, if we assumed that the singularity is of Type II, i.e. $\sup_{M \times [0,T)} (T-t)\abs{\Rm}_{g(t)} = \infty$, then there exists a sequence $t_i \to T$ such that $K_i(T-t_i) \to \infty$, and hence for any fixed $t$ we have
\[\abs{A}_{g_i(t)}^2 \to 0.\]
This shows the rescaled sequence $(M, g_i(t),x_i)$ converges (after passing to subsequences) in pointed Gromov-Hausdorff sense to a steady Ricci soliton $(M_\infty, g_\infty(t), x_\infty)$ with $\abs{A}_{g_\infty(t)} \equiv 0$. Together with $\abs{T}_{g_\infty(t)} \equiv 0$, the metric $g_\infty(t)$ splits isometrically into a product manifold $(\Sigma \times N', g_1(t) \oplus g_2(t))$, where $\dim_\C \Sigma = 1$.

When $\{X,Y\}$ are orthonormal horizontal frame, then from \eqref{eq:Rm(XYYX)} we have:
\[\Rm_g(X,Y,Y,X) = \sum_j \frac{1}{F_j^2}\Rm_{N_i}(F_jX_j,F_jY_j,F_jY_j,F_jX_j) + \sum_j O\left(\frac{H^2}{F_j^4}\abs{X}_g^2\abs{Y}_g^2\right)\]
where $X_i = d\pi_i(X)$ and $Y_i = d\pi_i(Y)$. For the K\"ahler case, we have
\[q_j H = 2F_jF_j' = \frac{\partial}{\partial s}F_j^2 \implies \frac{H^2}{F_j^4} = \frac{1}{q_j^2}\abs{\nabla\log F_j^2}^2 = \frac{4}{q_j^2}\frac{\abs{\nabla F_j}^2}{F_j^2}.\]

Now consider a time-dependent orthonormal (with respect to $g(t)$) horizontal frame $\{X(t),Y(t)\}$. Express them as
\begin{align*}
X(t) & = \sum_j X_j(t), & Y(t) & = \sum_j Y_j(t)	
\end{align*}
where $X_j(t) = d\pi_j(X(t))$ and $Y_j(t) = d\pi_j(Y(t))$.

Now consider the rescaled frame which is orthonormal with respect to the rescaled metric $g_i(t)$:
\begin{align*}
X^{(i)}(t) & := \frac{X(t_i+K_i^{-1}t)}{\abs{X(t_i+K_i^{-1}t))}_{g_i(t)}} = \frac{1}{\sqrt{K_i}}\frac{X}{\abs{X}_{g(t_i+K_i^{-1}t)}}\\
Y^{(i)}(t) & := \frac{Y(t_i+K_i^{-1}t)}{\abs{Y(t_i+K_i^{-1}t))}_{g_i(t)}} = \frac{1}{\sqrt{K_i}}\frac{Y}{\abs{Y}_{g(t_i+K_i^{-1}t)}}.
\end{align*}
Here we abbreviate $X(t_i+K_i^{-1}t)$ by simply $X$ and $Y(t_i+K_i^{-1}t)$ by simply $Y$.

Then, we have
\begin{align*}
& \Rm_{g_i(t)}\big(X^{(i)}(t),Y^{(i)}(t),Y^{(i)}(t),X^{(i)}(t)\big)\\
& = K_i \Rm_{g(t_i+K_i^{-1}t)}\big(X^{(i)}(t),Y^{(i)}(t),Y^{(i)}(t),X^{(i)}(t)\big)\\
& = \frac{1}{K_i}\Rm_{g(t_i+K_i^{-1}t)}\left(\frac{X}{\abs{X}_{g(t_i+K_i^{-1}t)}},\frac{Y}{\abs{Y}_{g(t_i+K_i^{-1}t)}},\frac{Y}{\abs{Y}_{g(t_i+K_i^{-1}t)}},\frac{X}{\abs{X}_{g(t_i+K_i^{-1}t)}}\right)\\
& = \frac{1}{K_i}\frac{1}{\abs{X}^2_{g(t_i+K_i^{-1}t)}}\frac{1}{\abs{Y}^2_{g(t_i+K_i^{-1}t)}} \cdot \\
& \hskip 1cm\left\{\sum_j \frac{1}{F_j^2}\Rm_{N_j}(F_jd\pi_j(X),F_jd\pi_j(Y),F_jd\pi_j(Y),F_jd\pi_j(X))\right.\\
& \hskip 3cm \left. + \sum_j O\left(\frac{1}{q_j^2}\frac{\abs{\nabla F_j}^2}{F_j^2}\abs{X}^2_{g(t_i+K_i^{-1}t)}\abs{Y}^2_{g(t_i+K_i^{-1}t)}\right) \right\}.
\end{align*}
Here we abbreviate $F_j(t_i+K_i^{-1}t)$ by simply $F_j$.

For each $j$, there exists a dimensional constant $C_j > 0$ such that
\begin{align*}
& \abs{\Rm_{N_j}(F_jd\pi_j(X),F_jd\pi_j(Y),F_jd\pi_j(Y),F_jd\pi_j(X)}\\
& \leq C_j F_j^4 \abs{d\pi_j(X)}_{g_{N_j}}^2 \abs{d\pi_j(Y)}_{g_{N_j}}^2\\
& = C_j \abs{X}_{g(t_i+K_i^{-1}t)}^2\abs{Y}_{g(t_i+K_i^{-1}t)}^2
\end{align*}
by the Riemannian submersion property. 

Furthermore, \eqref{eq:gradu/u} implies $\abs{\nabla F_j} \leq C$. Combining all these estimates, we get:
\[\abs{\Rm_{g_i(t)}\big(X^{(i)}(t),Y^{(i)}(t),Y^{(i)}(t),X^{(i)}(t)\big)} \leq O\left(\frac{1}{K_iF_j^2(t_i+K_i^{-1}t)}\right).\]
By \eqref{eq:Schwarz}, we have
\[F_j^2(t_i+K_i^{-1}t) \geq \frac{T-(t_i + K_i^{-1}t)}{C} \implies K_iF_j^2(t_i+K_i^{-1}t) \geq \frac{K_i(T-t_i) - t}{C}.\]
As we assume Type II singularity, we have $K_i(T-t_i) - t \to \infty$ as $i \to \infty$. We conclude that $\Rm_{g_\infty(t)}(X,Y,Y,X) = 0$ for any horizontal vectors $X, Y \in TN'$, show the $(N', g_2(t))$-factor is flat.

By the split structure of $(M_\infty,g_\infty(t)) = (\Sigma \times N', g_1(t) \oplus g_2(t))$, and the fact that $(N',g_2(t))$ is flat and $\dim_{\R}\Sigma = 2$, the absolute-value of the scalar curvature $\abs{R_{g_\infty(t)}}$ then equals $\abs{\Rm}_{g_\infty(t)}$ (up to a constant factor). Also, the limit metric $g_\infty(t)$ is defined on $(-\infty, \infty)$ since $g_i(t)$ is defined on $[-K_it_i, K_i(T-t_i))$ and we have $K_i(T-t_i) \to \infty$. By standard parabolic maximum principle, any ancient solution to the Ricci flow must have either positive, or identically zero scalar curvature. The later case is ruled out since $\abs{\Rm}_{g_\infty(t)}$ at $x_\infty$ is not zero. This shows $(\Sigma, g_1(t))$ must have positive scalar curvature, and is a steady Ricci soliton. However, it shows $(\Sigma, g_1(t))$ must be a cigar soliton by the standard classification of two-dimensional steady Ricci soliton, which is impossible as a limit metric by the local non-collapsing theorem due to Perelman. This concludes that the flow $(M,g(t))$ must encounter Type I singularity in the case $\inf_{M \times [0,T)} u(x,t) > 0$.
\end{proof}

\section{Discussion on Blow-up Models of Singularity}
To further deepen our understanding of the singularity formation, it would be interesting, as in former works \cite{Fong14,Song15,GS16,jst23}, to study the pointed Cheeger-Gromov limit of the rescaled sequence
\[g_i(t) := K_i g(t_i + K_i^{-1}t).\]
Knowing that the singularity is of Type I, we can assume without loss of generality that
\[K_i = \frac{1}{T-t_i}.\]
Perelman's local non-collapsing theorem \cite{perelman1} and Hamilton-Cheeger-Gromov's compactness \cite{Hcomp} show that there exist diffeomorphisms $\Phi_i$'s such that $(M, K_i \Phi_i^*g_i(t))$ converges (after passing to subsequences) to a limit manifold and Ricci flow solution $(M, g_\infty(t))$.

In the current literature, the singularity formation for $\CP^1$-bundles over K\"ahler-Einstein manifolds with Calabi's symmetry can be divided into two major cases, namely (i) collapsing of $\CP^1$-fibers, or (ii) contraction of the zero/infinity-sections. Using the form of metrics in this article, i.e.
\[g(t) = ds^2 + H(s,t)^2 \eta \otimes \eta + F(s,t)^2 \pi^*g_N,\]
the fiber-collapsing case corresponds to the situation that $\displaystyle{\lim_{t \to T} \sup H(\cdot, t) = 0}$ and $\displaystyle{\inf F(\cdot,\cdot) > 0}$, whereas the contracting of zero/infinity-sections case corresponds to $\displaystyle{\lim_{t \to T} \inf F(\cdot, t) = 0}$ and $\displaystyle{\inf H(\cdot,\cdot) > 0}$.

Now consider $\CP^1$-bundles over a \emph{product} of K\"ahler-Einstein manifolds $(N := N_1 \times \cdots \times N_r, \oplus_k \pi_k^* g_{N_k})$ with metrics constructed by the ansatz \eqref{eq:DW}:
\[g(t) = ds^2 + H(s,t)^2 \eta \otimes \eta + \sum_{k=1}^r F_k(s,t)^2 \pi_k^* g_{N_k}, \quad s \in [s_0, s_1], t \in [0,T).\]
We expect that the singularity formation in the K\"ahler case will be divided into three major cases:
\begin{enumerate}[(i)]
	\item collapsing of $\CP^1$-fibers, where $H$ converges to $0$ as $t \to T$ while each $F_k$ is bounded away from $0$,
	\item contraction of the entire zero/infinity section, where $H$ is bounded away from $0$ as $t \to T$, while each $F_k(s_0, t) \to 0$ (if each $q_k > 0$), or each $F_k(s_1, t) \to 0$ (if each $q_k < 0$),
	\item contraction of \emph{some} but not all $g_N$-components of the zero/infinite-section, which will be discussed in more detail below.
\end{enumerate}

Given that we have proved that the singularity is of Type I, the singularity model of case (i) is the direct product $\CP^1 \times \C^{\dim_\C N}$ as in the case of \cite{Fong14}. The proof is similar to that in \cite{Fong14} (see also the conformal submersion adoptation in \cite{Hoan}). The key idea is that we assume in (i) that each $F_k$ is bounded away from $0$, so by Corollary \ref{cor:|A|^2<|grad u|^2}, we have
\[\abs{A}_{g(t)}^2 \leq C\sum_k \abs{\nabla\log F_k}^2 \leq \sum_k O\left(\abs{\nabla F_k^2}^2\right).\]
As each $F_k^2$ satisfies \eqref{eq:gradu}, so each $\abs{\nabla F_k^2}$ is uniformly bounded. These shows $\abs{A}_{g(t)}^2$ is uniformly bounded and hence $\abs{A}_{g_i(t)}^2 = \frac{1}{K_i^2}\abs{A}^2_{g(t_i + K_i^{-1}t)} \to 0$ as $i \to \infty$. This shows the limit model $(M_\infty, g_\infty(t))$ splits into a product of $(\Sigma_1^2 \times \Sigma_2^{2\dim_\C N}, g_1(t) \oplus g_2(t))$. Similar to \cite{Fong14}, we can then identify that $(\Sigma_2, g_2(t))$ is flat by the expression of Riemann curvature \eqref{eq:Rm(XYYX)} which shows, under the fiber-collapsing assumption, $\abs{\Rm_{g(t)}(X,Y,Y,X)} \leq C\abs{X}^2\abs{Y}^2$ for any horizontal vectors $X$ and $Y$ before rescaling, and hence after rescaling the horizontal component of $\Rm_{g_{\infty(t)}}$ would vanish.

For the volume non-collapsing cases (ii) and (iii), we expect that case (ii), where all of $F_k(s_0, t)$'s converge to $0$ as $t \to T$, would be similar to the contracting divisor case in \cite{GS16,jst23} where the whole zero section contracts to a point. By adopting the proofs in \cite{GS16,jst23}, it should be possible to show that the blow-up limit should be a shrinking Ricci soliton on the total space of a line bundle over $N$ constructed by Dancer-Wang in \cite{DW11coho}. However, case (iii), where some but not all of $F_k(s_0,t)$'s converge to $0$ while others remain positive at $t = T$, seems to be more subtle and the current literature does not seem to give any suggestion on what kind of blow-up limits will be obtained. One possibility is that it would be a shrinking Ricci soliton on the total space of a line bundle over a product of some compact K\"ahler-Einstein manifolds and a flat space, or on the product of a flat space and the total space of a line bundle over a product of compact K\"ahler-Einstein manifolds. It would be an interesting direction for further research.

\section{Appendix}
\begin{proof}[Proof of Lemma \ref{lm1}]
	First for a submersion $\pi$ and the fact that $\pi_*V = 0$, we have
	\[\pi_*[X,V] = [\pi_*X, \pi_*V] = 0.\]
	Therefore, $[X,V] \in \ker\pi_* = \V$. Then, by the Koszul's formula, we have
	\begin{align*}
		2g(\nabla_X Y, V) & = X\big(g(Y,V)\big) + Y\big(g(X,V)\big) - V\big(g(X,Y)\big)\\
		& \hskip 0.5cm + g([X,Y],V) - g([X,V],Y) - g([Y,V],X)\\
		& = -V\left(\sum_i F_i^2 \pi_i^* g_{N_i}(X,Y)\right) + g([X,Y],V)\\
		& = -\sum_i g(\nabla F_i^2, V) \pi_i^*g_{N_i}(X,Y) + g([X,Y],V).
	\end{align*}
	Since the above holds for all $V \in \V$, we have
	\begin{equation}
		\label{eq:Vnabla_XY}
		2\V(\nabla_X Y) = - \sum_i \V(\nabla F_i^2) \pi_i^*g_{N_i}(X,Y) + \V[X,Y].
	\end{equation}
	This proves the first equality since $A_XY = \V(\nabla_X Y)$ for any $X, Y \in \H$. For the second equality, we relate $\V [X,Y]$ with $d\omega$ and consider
	\begin{align*}
		d\omega(X,Y,JV) & = X\big(\omega(Y,JV)\big) - Y\big(\omega(X,JV)\big) + JV\big(\omega(X,Y)\big)\\
		& \hskip 0.5cm -\omega\big([X,Y],JV\big) + \omega\big([X,JV],Y\big) - \omega\big([Y,JV],X\big)\\
		& = X\big(g(JY,JV)\big) - Y\big(g(JX,JV)\big) + JV\big(g(JX,Y)\big)\\
		& \hskip 0.5cm - g\big(J[X,Y],JV\big) + g\big(J[X,JV],Y\big) - g\big(J[Y,JV],X\big)\\
		& = JV\big(g(JX,Y)\big) - g\big(J[X,Y],JV\big)\\
		& = JV\left(\sum_{i=1}^r F_i^2 \pi_i^*g_{N_i}(JX,Y)\right) - g\big([X,Y], V\big).
	\end{align*}
	Therefore, we have
	\begin{align*}
		g\big([X,Y],V\big) & = \sum_{i=1}^r g\big(\nabla F_i^2, JV\big) \pi_i^*g_{N_i}(JX,Y)\big) - d\omega(X,Y,JV)\\
		& = -2\sum_{i=1}^r g\big(J\nabla \log F_i, V) F_i^2 \pi_i^*\omega_{N_i}(X,Y) - d\omega(X,Y,JV),
	\end{align*}
	and hence
	\begin{equation}
		\label{eq:V[X,Y]}
		\V[X,Y] = - 2\sum_{i=1}^r F_i^2 \pi_i^*\omega_{N_i}(X, Y) \V J\nabla \log F_i -\V\left(d\omega(X,Y,J(\cdot))\right)^\sharp.
	\end{equation}
	It completes the second equality.
	
	Next, we take any $Y \in \H$ and consider
	\begin{align*}
		g\big(\nabla_X V, Y\big) & = X\big(g(V,Y)\big) - g(V, \nabla_X Y)\\
		& = -g(V, A_XY).
	\end{align*}
	Hence, we have $\H\nabla_XV = -g\big(V, A_X(\cdot)\big)^\sharp$. The last equality is trivial.
\end{proof}
\begin{proof}[Proof of Corollary \ref{cor:|A|^2<|grad u|^2}]
	It follows from the fact that if $\{e_i\}$ is an orthonormal (with respect to $g$) frame of $\H$, then
	\begin{align*}
		\abs{F_k^2\pi_k^*\omega_{N_k}(e_i,e_j)} & = \abs{F_k^2 \pi_k^*g_{N_k}(Je_i,e_j)}\\
		& \leq F_k^2 \abs{d\pi_k(Je_i)}_{g_{N_k}}\abs{d\pi_k(e_j)}_{g_{N_k}}\\
		& \leq \abs{d\pi_k(Je_i)}_g \abs{d\pi_k(e_j)}_g\\
		& \leq \abs{Je_i}_g \abs{e_j}_g = 1.
	\end{align*}
	Similarly, we have
	\[\abs{F_k^2\pi_k^*\omega_{N_k}(e_i,e_j)} \leq 1.\]
	The corollary then follows from Lemma \ref{lm1}.
\end{proof}

\begin{proof}[Proof of Lemma \ref{computeAmul}]
	As the tensor $\tilde{A}$ are related to the tensor $A$ by $\tilde{A} = \textup{proj}_{TP_s}(A)$, and that $\nabla \log F_i \in \text{span}\{\nu\}$, by Lemma \ref{lm1} we have
	\[\tilde{A}_{X}Y = \frac{1}{2}\tilde{\V}[X, Y].\]
	Taking the exterior derivative of $\eta^*$, we get:
	\begin{align*}
		d\eta^\ast(X_i, Y) &= (\tilde{\nabla}_{X_i} \eta^\ast)Y-(\tilde{\nabla}_Y \eta^\ast) X_i\\
		&= \tilde{\nabla}_{X_i} (\eta^\ast(Y))-\eta^\ast(\tilde{\nabla}_{X_i} Y)-\tilde{\nabla}_Y(\eta^\ast(X_i))+\eta^\ast(\tilde{\nabla}_Y X_i)\\
		&= -\eta^\ast([X_i, Y])=-g_s(\zeta^\ast, [X_i, Y]) = -2g_s(\tilde{A}_{X_i}Y,\zeta^*).
	\end{align*}
	On the other hand, as $d\eta = \sum_i q_i\pi_i^*\omega_{N_i}$ and $dH(X_i) = dH(Y) = 0$, we have
	\begin{align*}
		d\eta^\ast (X_i, Y)&= (dH \wedge \eta + Hd\eta)(X_i, Y)=H q_i \omega_{N_i}(X_i, Y)\\
		& = \frac{q_iH}{F_i^2}g(JX_i,Y)
	\end{align*}
	completing the proof of the first equality.
	
	The second equality follows from the fact that $g_s(\tilde{A}_X \zeta^\ast, Y) =-g_s(\zeta^\ast, \tilde{A}_X Y)$.

\end{proof}
\begin{proof}[Proof of Proposition \ref{Rcdeformed1}]
	From \cite[Cor. 2.5.17]{BGbookSasakian08}, we have
	\[K_s(X_i,Z_i) = K_{N_i}(X_i,Z_i) - 3\abs{\tilde{A}_{X_i}Z_i}_{g_s}^2.\]
	Note that $\tilde{A}_{X_i}Z_i$ has only $\zeta^*$-component, we have
	\[\abs{\tilde{A}_{X_i}Z_i}^2_{g_s} = g(\tilde{A}_{X_i}Z_i,\zeta^*)^2 = \frac{q_i^2H^2}{4F_i^4} g_s(J_{N_i}X_i,Z_i)^2\]
	from Lemma \ref{computeAmul}, completing the proof of the first result.
	
	Similarly, using the fact that $T = 0$ under the ansatz \eqref{eq:DW}, Corollary 2.5.17 in \cite{BGbookSasakian08} shows
	\[K_s(X_i,\zeta^*) = \abs{\tilde{A}_{X_i}\zeta^*}^2 = \abs{\frac{q_iH}{2F_i^2}JX_i}^2 = \frac{q_i^2H^2}{4F_i^4}\abs{X_i}^2_{g_s}\]
	proving the second result.
	
	From \cite[Theorem 2.5.18]{BGbookSasakian08} and noting that $T = 0$ under ansatz \eqref{eq:DW}, we have
	\begin{align*}
		\widetilde{\Ric}(X_i,Y_i) & = \Ric_{N_i}(X_i,Y_i) - 2g_s(\tilde{A}_{X_i}, \tilde{A}_{Y_i})\\
		& = k_ig_{N_i}(X_i,Y_i)- 2g_s(\tilde{A}_{X_i}, \tilde{A}_{Y_i})\\
		& = \frac{k_i}{F_i^2}g_s(X_i,Y_i) - 2g_s(\tilde{A}_{X_i}, \tilde{A}_{Y_i}).
	\end{align*}
	Here we have used the fact that $(N_i,g_{N_i})$ is a K\"ahler-Einstein manifold. From Lemma \ref{computeAmul}, we have
	
	\begin{align*}
		g_s(\tilde{A}_{X_i}, \tilde{A}_{Y_i}) & = \left(\frac{q_iH}{2F_i^2}\right)^2 g_s(JX_i,JY_i) = \frac{q_i^2H^2}{4F_i^4} g_s(X_i,Y_i)
	\end{align*}
	completing the proof of the third equation.
	
	Finally, from \cite[Theorem 2.5.18]{BGbookSasakian08} again and observing that $T = 0$ and fibers of $\pi : P_s \to N$ is 1-dimensional (hence its Ricci curvature vanishes), we have
	\[\widetilde{\Ric}(\zeta^*,\zeta^*) = g_s(\tilde{A}\zeta^*, \tilde{A}\zeta^*)\]
	Suppose $\{e_a, Je_a\}$ is an orthonormal basis for $\title{\H}$ such that each $e_a$ is in exactly one of $TN_i$, then we have
	\begin{align*}
		& g_s(\tilde{A}\zeta^*,\tilde{A}\zeta^*)\\
		& = \sum_a \left(\abs{\tilde{A}_{e_a}\zeta^*}_{g_s}^2+\abs{\tilde{A}_{Je_a}\zeta^*}_{g_s}^2\right)\\
		& = \sum_i \sum_{e_a \in TN_i}\left(\abs{\tilde{A}_{e_a}\zeta^*}_{g_s}^2+\abs{\tilde{A}_{Je_a}\zeta^*}_{g_s}^2\right)\\
		& = \sum_i 2n_i \left(\frac{q_iH}{2F_i^2}\right)^2
	\end{align*}
	proving the last equality.
\end{proof}

\bibliographystyle{alpha}
\bibliography{bioMorse}

\def\cprime{$'$}
\begin{thebibliography}{Ham95b}

\bibitem[App23]{Appleton23}
Alexander Appleton.
\newblock Eguchi--hanson singularities in u (2)-invariant ricci flow.
\newblock {\em Peking Mathematical Journal}, 6(1):1--141, 2023.

\bibitem[BG08]{BGbookSasakian08}
Charles~P. Boyer and Krzysztof Galicki.
\newblock {\em Sasakian geometry}.
\newblock Oxford Mathematical Monographs. Oxford University Press, Oxford,
  2008.

\bibitem[CHM25]{CHM25typeI}
Ronan~J Conlon, Max Hallgren, and Zilu Ma.
\newblock Non-collapsed finite time singularities of the ricci flow on compact
  k$\backslash$" ahler surfaces are of type i.
\newblock {\em arXiv preprint arXiv:2502.19804}, 2025.

\bibitem[DG20]{Francesco20}
Francesco Di~Giovanni.
\newblock Ricci flow of warped berger metrics on r 4.
\newblock {\em Calculus of Variations and Partial Differential Equations},
  59(5):162, 2020.

\bibitem[DW11]{DW11coho}
Andrew~S. Dancer and McKenzie~Y. Wang.
\newblock On {R}icci solitons of cohomogeneity one.
\newblock {\em Ann. Global Anal. Geom.}, 39(3):259--292, 2011.

\bibitem[EMT11]{emt10}
Joerg Enders, Reto M{\"u}ller, and Peter~M. Topping.
\newblock On type-{I} singularities in {R}icci flow.
\newblock {\em Comm. Anal. Geom.}, 19(5):905--922, 2011.

\bibitem[FIK03]{fik03}
Mikhail Feldman, Tom Ilmanen, and Dan Knopf.
\newblock Rotationally symmetric shrinking and expanding gradient
  {K}\"{a}hler-{R}icci solitons.
\newblock {\em J. Differential Geom.}, 65(2):169--209, 2003.

\bibitem[Fon14]{Fong14}
Frederick Tsz-Ho Fong.
\newblock K\"ahler-{R}icci flow on projective bundles over
  {K}\"ahler-{E}instein manifolds.
\newblock {\em Trans. Amer. Math. Soc.}, 366(2):563--589, 2014.

\bibitem[GS16]{GS16}
Bin Guo and Jian Song.
\newblock On {F}eldman-{I}lmanen-{K}nopf's conjecture for the blow-up behavior
  of the {K}\"ahler {R}icci flow.
\newblock {\em Math. Res. Lett.}, 23(6):1681--1719, 2016.

\bibitem[Ham82]{H3}
Richard~S. Hamilton.
\newblock Three-manifolds with positive {R}icci curvature.
\newblock {\em J. Differential Geom.}, 17(2):255--306, 1982.

\bibitem[Ham95a]{Hcomp}
Richard~S. Hamilton.
\newblock A compactness property for solutions of the {R}icci flow.
\newblock {\em Amer. J. Math.}, 117(3):545--572, 1995.

\bibitem[Ham95b]{Hsurvey}
Richard~S. Hamilton.
\newblock The formation of singularities in the {R}icci flow.
\newblock In {\em Surveys in differential geometry, Vol.\ II (Cambridge, MA,
  1993)}, pages 7--136. Internat. Press, Cambridge, MA, 1995.

\bibitem[IKMS25]{IKS25}
James Isenberg, Dan Knopf, Zilu Ma, and Natasa Sesum.
\newblock Local singularities of compact multiply warped ricci flow solutions.
\newblock {\em arXiv preprint arXiv:2502.08500}, 2025.

\bibitem[JST23]{jst23}
Wangjian Jian, Jian Song, and Gang Tian.
\newblock Finite time singularities of the {K}\"{a}hler-{R}icci flow.
\newblock {\em preprint, 2310.07945}, 10 2023.

\bibitem[LTZ24]{LTZ24singular}
Yan Li, Gang Tian, and Xiaohua Zhu.
\newblock Singular limits of k{\"a}hler-ricci flow on fano g-manifolds.
\newblock {\em American Journal of Mathematics}, 146(6):1651--1690, 2024.

\bibitem[LY86]{ly86}
Peter Li and Shing~Tung Yau.
\newblock On the parabolic kernel of the {S}chr\"odinger operator.
\newblock {\em Acta Math.}, 156(3-4):153--201, 1986.

\bibitem[Max14]{Maximo14}
Davi Maximo.
\newblock On the blow-up of four-dimensional {R}icci flow singularities.
\newblock {\em J. Reine Angew. Math.}, 692:153--171, 2014.

\bibitem[Ngu23]{Hoan}
The~Hoan Nguyen.
\newblock K\"{a}hler-{R}icci flow and conformal submersion.
\newblock {\em Mediterranean Journal of Mathematics}, 20(5), July 2023.

\bibitem[O'N66]{ONeill66}
Barrett O'Neill.
\newblock The fundamental equations of a submersion.
\newblock {\em Michigan Math. J.}, 13:459--469, 1966.

\bibitem[Per02]{perelman1}
Grisha Perelman.
\newblock The entropy formula for the {R}icci flow and its geometric
  applications.
\newblock {\em arXiv:math.DG/0211159}, 2002.

\bibitem[Per03a]{perelman3}
Grisha Perelman.
\newblock Finite extinction time for the solutions to the {R}icci flow on
  certain three-manifolds.
\newblock {\em arXiv:math.DG/0307245}, 2003.

\bibitem[Per03b]{perelman2}
Grisha Perelman.
\newblock Ricci flow with surgery on three-manifolds.
\newblock {\em arXiv:math.DG/0303109}, 2003.

\bibitem[Son15]{Song15}
Jian Song.
\newblock Some type {I} solutions of {R}icci flow with rotational symmetry.
\newblock {\em Int. Math. Res. Not. IMRN}, (16):7365--7381, 2015.

\bibitem[ST17]{ST17KRf}
Jian Song and Gang Tian.
\newblock The {K}{\"a}hler--{R}icci flow through singularities.
\newblock {\em Inventiones mathematicae}, 207(2):519--595, 2017.

\bibitem[SW11]{SW11}
Jian Song and Ben Weinkove.
\newblock The {K}\"{a}hler-{R}icci flow on {H}irzebruch surfaces.
\newblock {\em Journal f{\"u}r die reine und angewandte Mathematik (Crelles
  Journal)}, 2011(659), January 2011.

\bibitem[SW13a]{SW13ii}
Jian Song and Ben Weinkove.
\newblock Contracting exceptional divisors by the {K}\"{a}hler-{R}icci flow.
\newblock {\em Duke Mathematical Journal}, 162(2), February 2013.

\bibitem[SW13b]{SW13i}
Jian Song and Ben Weinkove.
\newblock Contracting exceptional divisors by the {K}\"{a}hler-{R}icci flow ii.
\newblock {\em Proceedings of the London Mathematical Society},
  108(6):1529--1561, November 2013.

\bibitem[SY12]{SY12}
Jian Song and Yuan Yuan.
\newblock Metric flips with calabi ansatz.
\newblock {\em Geometric and Functional Analysis}, 22(1):240--265, January
  2012.

\bibitem[TZ15]{TZ15}
Valentino Tosatti and Yuguang Zhang.
\newblock Finite time collapsing of the k{\"a}hler-ricci flow on threefolds.
\newblock {\em Ann. Sc. Norm. Super. Pisa Cl. Sci. (5) 18 (2018), no.1,
  105-118}, 07 2015.

\bibitem[Wu25]{wu25asymptotic}
Haotian Wu.
\newblock Asymptotic analysis of {R}icci flow on {$\Bbb R^{n+1}$} with
  type-{II}b singularities.
\newblock {\em J. Geom. Anal.}, 35(6):Paper No. 174, 25, 2025.

\bibitem[WW98]{WW98}
Jun Wang and McKenzie~Y. Wang.
\newblock Einstein metrics on {$S^2$}-bundles.
\newblock {\em Math. Ann.}, 310(3):497--526, 1998.

\end{thebibliography}

\end{document}